\documentclass{article}
\usepackage[utf8]{inputenc}
\usepackage{amsmath}
\usepackage{amsthm}
\usepackage{amsfonts}
\usepackage{amssymb}
\usepackage{latexsym}
\usepackage{graphicx}
\graphicspath{{Images/}}
\numberwithin{equation}{section}
\title{Resolvent Estimates for Non-Self-Adjoint Magnetic Schr\"{o}dinger Operators}
\author{Ben Bellis\\Department of Mathematics\\UCLA\\Los Angeles, CA, 90095\\bbellis@math.ucla.edu}
\date{}
\newtheorem{theorem}{Theorem}[section]
\newtheorem{lemma}[theorem]{Lemma}

\newtheorem{proposition}[theorem]{Proposition}
\newtheorem{cor}[theorem]{Corollary}
\renewcommand{\(}{\left(}
\renewcommand{\)}{\right)}
\newcommand{\R}{\mathbb{R}}
\newcommand{\C}{\mathbb{C}}
\newcommand{\Sch}{\mathcal{S}}
\newcommand{\rr}{\textrm{Re}\,}
\newcommand{\im}{\textrm{Im}\,}
\newcommand{\supp}{\textrm{supp}}
\begin{document}
\maketitle
\begin{abstract}
    We examine semiclassical magnetic Schr\"{o}dinger operators with complex electric potentials.
    Under suitable conditions on the magnetic and electric potentials, 
    we prove a resolvent estimate for spectral parameters in an unbounded parabolic neighborhood of the imaginary axis.
\end{abstract}

\section{Introduction}
In this paper we study non-self-adjoint Schr\"{o}dinger operators with magnetic potentials. 
Non-self-adjoint Schr\"{o}dinger operators appear in a variety of contexts such as the study of resonances \cite{sjlec}, 
Hamiltonians of open systems \cite{ex}, and the damped wave equation \cite{csvw}.
Those with magnetic potential are of particular importance in Ginzburg-Landau theory in the study of superconductivity \cite{ahelf}. 
One difficulty of working with non-self-adjoint operators as opposed to the self-adjoint case is the lack of the spectral theorem. 
Whereas the size of the resolvent of a self-adjoint operator is a function of the distance between the spectral parameter and the operator's spectrum,
there is no analog for non-self-adjoint operators, for which the resolvent can grow large even far away from the spectrum \cite{dsz}, \cite{sjnotes}.
As such, it becomes of interest to study under what circumstances and in which regions of the complex spectral plane
we can establish useful estimates of the size of the resolvent for such operators.

Non-self-adjoint magnetic Schr\"{o}dinger operators have the form 
$$P=\(hD_{x}-A\(x\)\)^{2}+V\(x\),\quad x\in \R^{n},\ D_{x}:=-i\partial_{x}$$
where we will take $V=V_{1}+iV_{2}$, with $V_{1},V_{2}\in C^{\infty}\(\R^{n};\R\)$, $A\in C^{\infty}\(\R^{n};\R^{n}\)$, and $h>0$. 
We refer to $A$ as the magnetic potential and $V$ as the electric potential. We shall study 
such operators in the semiclassical limit, i.e. we shall be concerned with the behavior of the operator in the limit as $h\rightarrow 0$. 
In the context of quantum mechanics
$h$ represents Planck's constant and taking $h$ to be small models the situation where the data is large relative to the quantum scale. 
Our primary goal is to show a particular resolvent estimate
for a broad class of such operators which generalizes one proven in \cite{mine} for the $A=0$ case.

To study such an operator we will use methods involving pseudodifferential operators. For this we will primarily use the Weyl quantization.
For a symbol $a:\R^{2n}\rightarrow \C$, the Weyl quantization of $a$ is given by
\[a^{w}\(x,D_{x}\)u:=\frac{1}{\(2\pi\)^{n}}\int_{\R^{2n}}e^{i\(x-y\)\cdot \xi}a\(\frac{x+y}{2},\xi\)u\(y\)dy d\xi,\]
and the semiclassical Weyl quantization is given by
\[a_{h}^{w}u=a^{w}\(x,hD_{x}\)u:=\frac{1}{\(2\pi\)^{n}}\int_{\R^{2n}}e^{i\(x-y\)\cdot\xi}a\(\frac{x+y}{2},h\xi \)u\(y\)dy d\xi.\]
Using this we can write $P=p^{w}_{h}$ for $p\(x,\xi\)=|\xi-A\(x\)|^{2}+V\(x\)$.

Now let us introduce some notation that we will use throughout this paper. For $X\in\R^{2n}$, $X=\(x,\xi\)$ with $x,\xi\in\R^{n}$.
The notation ``$f\lesssim g$" means there exists a $C>0$, independent of $X$, $h$ and other parameters, such that $f\leq Cg$.
Additionally, the symbol class $S\(m\)$ is defined by 
\[S\(m\):=\{a\in C^{\infty}\(\R^{N}\): |\partial^{\alpha}a|\lesssim m,\quad \forall \alpha\},\]
for $m$ some positive function on $\R^{N}$.
We shall place the following condition on the symbol $p$:
\begin{equation}\label{eq:v1pos}
    V_{1}\geq 0
\end{equation}
\begin{equation}\label{eq:v''s1}
    V''\in S\(1\),
\end{equation}
\begin{equation}\label{eq:A'}
    |A'\(x\)|\lesssim 1,\quad x\in\R^{n},
\end{equation}
\begin{equation}\label{eq:A''}
    A''\in S\(\langle x\rangle ^{-1}\),
\end{equation}
\begin{equation}\label{eq:vv'}
    |V_{2}\(x\)|\lesssim 1+V_{1}\(x\)+|V_{2}'\(x\)|^{2},\quad x\in\R^{n}.
\end{equation}
One implication of these conditions that will be useful is that by \eqref{eq:v1pos} and \eqref{eq:v''s1}
\begin{equation}\label{eq:v1'v1}
    |V_{1}'|\lesssim V_{1}^{1/2},
\end{equation}
as this holds for any nonnegative $C^{2}$ function with bounded second derivatives \cite{zw}.
It then follows that if we define
\[m_{p}\(X\):=1+\rr p\(X\) + |V_{2}\(x\)'|^{2},\]
these conditions collectively imply that
\begin{equation}\label{eq:symbolclass}
    p\in S\(m_{p}\).
\end{equation}
Let $P:=p^{w}\(x,hD_{x}\)$. When regarding $P$ as a closed, unbounded operator on $L^{2}$, we equip $P$ with the maximal domain
$D\(P\):=\{u\in L^{2}\(\R^{n}\): p_{h}^{w}u\in L^{2}\}$. The following is the main result of this paper.
\begin{theorem}\label{thm1}
Let $T\geq 0$ be such that $|V_{2}|-T\lesssim V_{1}+|V_{2}'|^{2}$.
For such $p$ and any $K>1$ there exist constants $C_{0}, M,  h_{0}$ with $0<h_{0}, C_{0}\leq 1$ and $M\geq 2$ such that for all $0<h\leq h_{0}$
and $z\in\C$ with $|z|\geq KT+Mh$ and $\rr z\leq C_{0} h^{2/3}\(|z|-T\)^{1/3}$ $\(P-z\)^{-1}$ exists and we get
\begin{equation}\label{eq:pres}
    \|\(P-z\)^{-1}\|_{L^{2}\rightarrow L^{2}}\lesssim h^{-2/3}\(|z|-T\)^{-1/3}
\end{equation}
\end{theorem}
For convenience, we will write $y:=|z|-T$, and so the conclusion of the theorem can be rewritten as
\[\|\(P-z\)^{-1}\|_{L^{2}\rightarrow L^{2}}\lesssim h^{-2/3}y^{-1/3}\]
for $z\in \C$, $|z|\geq KT+Mh$, $\rr z\leq C_{0}h^{2/3}y^{1/3}$.
Also, note that when $T\neq 0$ the constant $M$ can be made irrelevant by taking $h_{0}$ small enough, and when $T=0$ the constant $K$ is irrelevant.
Theorem \ref{thm1} can thus be considered as applying to two distinct cases: when $T\neq 0$ and thus $|z|\geq KT$ and $y\geq \(K-1\)T\gtrsim 1$, 
and a sharper estimate when $T=0$ and so $|z|\geq Mh$ and $y=|z|$.

To prove this we start by proving the $L^{2}$ lower bound
\begin{equation}\label{eq:pest}
    \|\(P-z\)u\|\gtrsim h^{2/3}y^{1/3}\|u\|,\quad \forall u\in\Sch.
\end{equation}
We do this by showing that it suffices to prove the same estimate for the Weyl quantization of a modified symbol $q$,
obtained by cutting off the magnetic potential in the region where $|x|\ll |\xi|$, taking advantage of the ellipticity of $p$ in this region.
We then prove such an estimate for $q_{h}^{w}$ by constructing a weight function to use as a bounded multiplier 
and then using some symbol calculus of pseudodifferential operators,
following a method very similar to that used in \cite{mine}.
Then we show that this lower bound extends to the maximal domain of $P$ by using a graph closure argument, 
which then implies that the desired resolvent estimate \eqref{eq:pres} holds.

The plan of the paper is as follows. 
In section 2 we construct the modified symbol $q$ and show that working with $q_{h}^{w}-z$ suffices.
In section 3 we construct a weight function $g$ and show that it satisfies a list of properties which will be needed to use it as a bounded multiplier later.
In section 4 we review some of the important properties of the Wick quantization.
In section 5 we prove the desired $L^{2}$ lower bound for $q_{h}^{w}-z$ using the weight function from section 3 and pseudodifferential symbol calculus in both the Wick and Weyl quantizations.
This then implies that \eqref{eq:pest} holds due to section 2.
In section 6 we prove that for a class of symbols which includes $p$, the graph closure of the corresponding Weyl quantization on $\Sch$ has maximal domain. 
We then apply this to conclude that $P-z$ is invertible and attain the desired resolvent estimate.

The reason why we work with a modified symbol is because when deriving $L^2$ estimates using pseudodifferential symbol calculus 
it will be necessary to bound certain derivatives of the symbol
of our operator. Specifically, we will be using the Calder\'{o}n-Vaillancourt theorem which states that for $a\in S\(1\)$, 
\begin{equation}\label{eq:cv}
    \|a^{w}\|_{L^{2}\rightarrow L^{2}}\lesssim \sup\limits_{|\alpha|\leq Mn}\|\partial^{\alpha}a\|_{L^{\infty}},
\end{equation}
where $M>0$ is some global constant, see \cite{zw}.
An obstacle for using this when working with a Schr\"{o}dinger operator with magnetic potential is that 
any derivative of the symbol $p$ with respect to $x$ will have a term of the form $\xi\cdot\partial^{\alpha}A\(x\)$, 
which is unbounded in $\xi$, even when $A$ is compactly supported.
However, in the region where $\langle x\rangle \gtrsim |\xi|$, the condition \eqref{eq:A''} implies that such a term is bounded for $|\alpha|\geq 2$.
Thus it is preferable to work with a modified symbol, $q$, which has the 
magnetic potential cut off in the region where $|\xi|\gtrsim \langle x\rangle$. 

\section{Truncating the Magnetic Potential}
We are able to modify $p$ in the region where $|\xi|$ is much larger than $|x|$ while keeping the same the $L^{2}$ lower bound 
because the derivative bounds on $A$ and $V$ imply that $|p|\sim |\xi|^{2}\sim m_{p}$, and thus $p$ is elliptic in this region. 
Thus, if we change the symbol there while keeping this ellipticity we can expect the new symbol to behave similarly to $p$.

First we recall a couple standard facts of Weyl symbol calculus.
We say that a function $m > 0$ on $\R^{N}$ is an order function if $m\(X\)\lesssim \langle X-Y\rangle^{k} m\(Y\)$ for all $X,Y\in\R^{N}$ and some $k$ (see \cite{zw}).
Let $m_{1}, m_{2}$ be order functions on $\R^{2n}$, and let $f\in S\(m_{1}\), g\in S\(m_{2}\)$.
We then have that $fg, f\# g\in S\(m_{1}m_{2}\)$ with $f\# g$ defined by 
\begin{equation}\label{eq:wcomp}
f_{h}^{w}g_{h}^{w}=\(f\# g\)_{h}^{w}=\(fg + \frac{h}{2 i}\{f,g\} +h^{2}r\)_{h}^{w},
\end{equation}
where 
\[f\# g=e^{\frac{ih}{2}\(D_{\xi}\cdot D_{y}-D_{x}\cdot D_{\eta}\)}f\(x,\xi\)g\(y,\eta\)\bigg|_{\(y,\eta\)=\(x,\xi\)} \]
and
\begin{multline*}
r=-\frac{1}{4}\int_{0}^{1}\(1-t\)e^{\frac{ith}{2}\(D_{\xi}\cdot D_{y}-D_{x}\cdot D_{\eta}\)}\\
\(D_{\xi}\cdot D_{y}-D_{x}\cdot D_{\eta}\)^{2}f\(x,\xi\)g\(y,\eta\)dt \bigg|_{\(y,\eta\)=\(x,\xi\)}.
\end{multline*}
In particular, expanding to first order, given $f$ and $g$ with $f'\in S\(m_{1}\)$ and $g'\in S\(m_{2}\)$ it holds that
\begin{equation}\label{eq:wcomp1}
    f_{h}^{w}g_{h}^{w}=\(fg\)_{h}^{w}+h\(r_{1}\)_{h}^{w},
\end{equation}
for some $r_{1}\in S\(m_{1}m_{2}\)$. Both \eqref{eq:wcomp} and \eqref{eq:wcomp1} can also apply to the non-semiclassical Weyl quantization by taking $h=1$.

Let $\chi\in C_{c}^{\infty}\(\R^{n};[0,1]\)$ be a smooth function with $\chi\(x\)=1$ for all $|x|\leq 1$ and $\chi\(x\)=0$ for all $|x|\geq 2$.
We use $\chi$ to cut off the magnetic potential in the region where $|\xi|$ is large relative to $\langle x \rangle$. Define
$\chi_{t}\(X\):=\chi\(\frac{\xi}{t\langle x\rangle}\)$
and
\begin{equation}\label{eq:qsymbol}
    q\(X\):=|\xi-\chi_{2R}\(X\)A\(x\)|^{2}+V\(x\).
\end{equation}
where $R>0$ is chosen sufficiently large such that $|A\(x\)|^{2},|V\(x\)| \leq \frac{1}{16}R^2\langle x\rangle^{2}$. 

Such an $R$ exists because conditions \eqref{eq:v''s1} and\eqref{eq:A'} imply that $|A\(x\)|\lesssim \langle x\rangle$ and $|V\(x\)|\lesssim \langle x\rangle^{2}$.
Thus for $X\in \supp\(1-\chi_{R}\)$
we have $|\xi|\geq R\langle x\rangle$, and so $|\xi-A\(x\)| \geq \frac{3|\xi|}{4}\geq \frac{3R\langle x\rangle}{4}$ and 
\[|V\(x\)|\leq |\xi-A\(x\)|^{2}\sim |\xi|^{2}.\]
We then see that $|\xi|^{2}$ dominates the other terms of $\rr p$ on the support of $1-\chi_{R}$, and so
\begin{equation}\label{eq:pqell}
    \rr p\sim \rr q \sim m_{p}\(X\)\sim |\xi|^{2}\sim \langle X\rangle^{2},\quad X\in \supp\(1-\chi_{R}\).
\end{equation}
We can see from \eqref{eq:qsymbol} that $\supp\(p-q\)\subseteq\supp\(1-\chi_{2R}\)$. 
Thus $p$ and $q$ are both elliptic of order $m_{p}$ in the region in which $p\neq q$.

Furthermore in this region, $\rr p, \rr q \geq \frac{R^{2}}{2}\langle x\rangle^{2}$ while $|\im p|, |\im q|\leq \frac{R^{2}}{16} \langle x\rangle^{2}$,
so $\rr p\geq 8|\im p|$ and the same holds for $q$.
The conditions of Theorem \ref{thm1} also require that $\rr z\leq C_{0}h^{2/3}y^{1/3}$ and $|z|\geq Mh$ with $M\geq 2$, $C_{0}\leq 1$. Thus 
$\rr z\leq 2^{-2/3}|z|$ which implies that $\rr z\leq |\im z|$.
Thus for such $z$ we have that
for $X\in \supp\(1-\chi_{R}\)$, $p\(X\)$ and $q\(X\)$ lie in a closed cone disjoint from one which contains $z$, as shown in Figure \ref{zregion}.
It follows that $|p\(X\)-z|\sim |p\(X\)|+|z|$, and the same for $q$. 
From this and \eqref{eq:pqell}, we now have that:
\begin{equation}\label{eq:pqz}
    |p\(X\)-z|\sim |q\(X\)-z|\sim \langle X\rangle^{2}+|z|,\quad X\in \supp\(1-\chi_{R}\).
\end{equation}
\begin{figure}[ht]
\centering
\includegraphics[width=.4\textwidth]{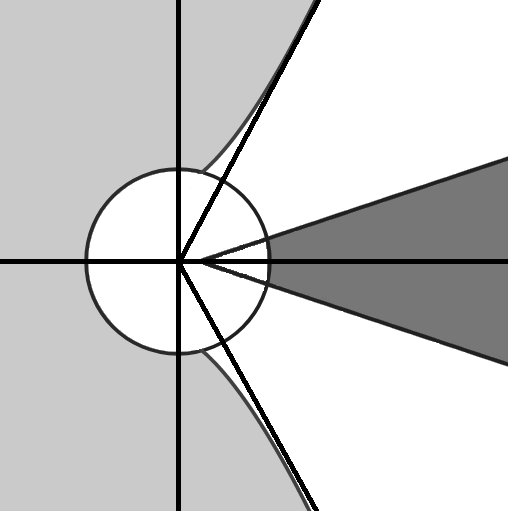}
\caption{The lighter shaded region indicates the values of $z$ for which Theorem \ref{thm1} applies. The ranges of $p\(X\)$ and $q\(X\)$ for $X\in \supp \(1-\chi_{R}\)$
lie within the darker shaded cone around the positive real axis.}
\label{zregion}
\end{figure}

What is convenient about working with this $q$, given in \eqref{eq:qsymbol} is that, unlike $p$, derivatives or order two and higher are bounded.
\begin{lemma}\label{q''lemma}
For $q$ as defined in \eqref{eq:qsymbol}, it holds that
    \begin{equation}\label{eq:q''s1}
        q''\in S\(1\).
    \end{equation}
\end{lemma}
\begin{proof}
Let us expand $q$ and consider each term.
\[q=|\xi|^{2}-2\xi\cdot A \chi_{2R} +\chi_{2R}^{2}|A|^{2}+V.\]
It is trivial that derivatives of order two and higher are bounded for $|\xi|^{2}$ and $V$, due to \eqref{eq:v''s1}.
To see that the same holds for the other two terms, we first observe that
    \begin{equation}\label{eq:chi'}
        |\partial^{\alpha}\chi_{t}\(X\)|\lesssim \langle X\rangle^{-|\alpha|},\quad |\alpha|\geq 0.
    \end{equation}
Then by \eqref{eq:A'} and \eqref{eq:chi'}, 
\[\left|\partial^{\alpha}\(A\(x\)\chi_{2R}\(X\)\)\right|\lesssim 1,\quad |\alpha|\geq 1,\]
and by \eqref{eq:A''} and \eqref{eq:chi'}
\[\left|\partial^{\alpha}\(A\(x\)\chi_{2R}\(X\)\)\right|\lesssim  \langle X\rangle^{-1},\quad |\alpha|\geq 2.\]
With these two estimates we can see that all derivatives of order at least two of $\xi\cdot A \chi_{2R}$ and $\chi_{2R}^{2}|A|^{2}$ are bounded.
\end{proof}
So that we may work with $q$ instead of $p$ we will establish that
\[\|\(q_{h}^{w}-z\)u\|\lesssim \|\(p_{h}^{w}-z\)u\|+O\(h\)\|u\|,\quad u\in\Sch,\ \rr z\leq |\im z|.\]
To see this we will use the following lemma.
\begin{lemma}\label{fs1}
    Define $F\(X\)$ by
    $$
    F(X) = \frac{q\(X\)-z}{p\(X\)-z}\(1-\chi_{R}\(X\)\)+\chi_{R}\(X\),
    $$
    where $z\in\{z\in\C: \rr z\leq |\im z|\}$.
    The symbol $F$ satisfies
    \[F\in S\(1\)\]
    and 
    \[F'\in S\(\langle X\rangle^{-1}\).\]
\end{lemma}
\begin{proof}
Recall that when we say $a\in S\(m\)$ for a symbol $a$ and symbol class $S\(m\)$, the implicit constants in the derivative bounds are independent of $z$ 
even when $a$ or $m$ depends on $z$.

To understand $F$, note that, by \eqref{eq:pqz}, both $p-z$ and $q-z$ are elliptic on $\supp\(p-q\)\subset \supp\(1-\chi_{R}\)$,
so whenever $p-z=0$ we also have that
$q-z=0$ and $F=1$.
Thus $F$ is equal to $\(q-z\)/\(p-z\)$ everywhere that the latter is defined, and $F\(X\)=1$ when $|\xi|\leq 2R\langle x\rangle$. In particular
\begin{equation}\label{eq:ptq}
    \(p-z\)F=q-z.
\end{equation}

We can then see that the bounds in the lemma are only nontrivial when $p\neq q$, 
i.e. for $X\in \supp\(1-\chi_{2R}\)$, and that 
$\supp\(F'\)\subseteq \supp\(1-\chi_{2R}\)$. 
In that region, \eqref{eq:pqz} implies that $|F|\lesssim 1$ uniformly in $z$ so it remains to check the size of the derivatives there.

By \eqref{eq:q''s1}, $q'\in S\(\langle X\rangle\)$ and $q\in S\(\langle X\rangle^{2}\)$.
Similarly, observe that the derivative bounds on $V$ and $A$, \eqref{eq:v''s1}, \eqref{eq:A'}, and \eqref{eq:A''}, imply the same for $p$:
\begin{equation}\label{eq:p'x}
    |\partial^{\alpha}p|\lesssim \langle X\rangle, \quad |\alpha|\geq 1.
\end{equation}
So we can say $p'\in S\(\langle X\rangle \)$, and $p\in S\(\langle X\rangle^{2}\)$ by \eqref{eq:symbolclass}.

Thus $p-z, q-z\in S\(\langle X\rangle^{2}+|z|\)$, and $|p-z|\gtrsim \langle X\rangle^{2}+|z|$ on $\supp\(1-\chi_{R}\)$ by \eqref{eq:pqz}.
It then follows from \eqref{eq:pqz} and \eqref{eq:chi'} that
\[\frac{1-\chi_{R}}{p-z}\in S\(\frac{1}{\langle X\rangle^{2}+|z|}\).\]
Using that $\chi_{R}=0$ on $\supp\(F'\)$ and $p=q$ on $\supp\(\chi_{R}'\)$, we have
\[F'=\(1-\chi_{R}\)\frac{\(p-z\)q'-\(q-z\)p'}{\(p-z\)^{2}}-\chi_{R}'+\chi_{R}'\]
\[=\(\frac{1-\chi_{R}}{p-z}\)^{2}\(\(p-z\)q'-\(q-z\)p'\).\]
Multiplying the symbol classes of these factors together yields
\[F'\in S\(\frac{1}{\(\langle X\rangle^{2}+|z|\)^{2}}\(\langle X\rangle^{2}+|z|\)\langle X\rangle\)\subseteq S\(\langle X\rangle^{-1}\).\]
As we have already shown that $|F|\lesssim 1$, this implies that $F\in S\(1\)$.
\end{proof}
This leads to the following.
\begin{cor}\label{qthenp}
Suppose it is true that
\begin{equation}\label{eq:qest}
    \|\(q_{h}^{w}-z\)u\|\gtrsim h^{2/3}y^{1/3}\|u\|,\quad u\in\Sch,
\end{equation}
for $z\in \C$ satisfying the hypotheses of Theorem \ref{thm1}.
Then
\[ \|\(p_{h}^{w}-z\)u\|\gtrsim h^{2/3}y^{1/3}\|u\|.\]
\end{cor}
\begin{proof}
As $F\in S\(1\)$ by Lemma \ref{fs1}, it follows from the Calder\'{o}n-Vaillancourt Theorem \eqref{eq:cv} that
\[\|F^{w}_{h}\(p^{w}_{h}-z\)u\|\lesssim \|\(p^{w}_{h}-z\)u\|.\]
Then because $F'\in S\(\langle X\rangle^{-1}\)$ and $p'\in S\(\langle X\rangle\)$, as noted in Lemma \ref{fs1}, using \eqref{eq:wcomp1}, \eqref{eq:ptq} and \eqref{eq:cv} yields
\[\|F^{w}_{h}\(p^{w}-z\)u\|\geq\|\(F\(p-z\)\)^{w}_{h}u\|-O\(h\)\|u\|=\|\(q^{w}_{h}-z\)u\|-O\(h\)\|u\|.\]
Thus
\[\|\(p^{w}_{h}-z\)u\|\gtrsim \|\(q^{w}_{h}-z\)u\|-O\(h\)\|u\|\gtrsim \(h^{2/3}y^{1/3}-O\(h\)\)\|u\|.\]
We required that $y\geq Mh$ for some large $M>0$, so taking $M$ large enough we get the desired estimate.
\end{proof}
Thus when proving Theorem \ref{thm1} we can get the desired lower bound for $P$, \eqref{eq:pest}, by showing that \eqref{eq:qest} holds. So our 
goal now is to show \eqref{eq:qest}.

\section{The Weight Function}
Note that for the original symbol $p$ we have the subellipticity property
\[\rr p + H_{\im p}^{2}\rr p = |\xi-A|^{2}+V_{1}+2|V_{2}'|^{2}\geq 0.\]
We use this as a basis for constructing a weight function $g$, to be used to form a bounded multiplier.

Let 
\[\lambda_{p}:=|\xi-A|^{2}+V_{1}+2|V_{2}'|^{2}=\rr p + 2|V_{2}'|^{2},\]
and 
\[\lambda_{q}:=\left|\xi-\chi_{2R}A\right|^{2}+V_{1}+2|V_{2}'|^{2}=\rr q +2|V_{2}'|^{2}.\]
Let $\psi\in C_{c}^{\infty}\(\R;[0,1]\)$ be a cutoff function with $\psi\(t\)=1$ for $|t|\leq \frac{1}{2}$ and $\psi\(t\)=0$ for $|t|\geq 1$.
\begin{lemma}\label{gexists}
There exists a function $g\in C^{\infty}\(\R^{2n};\R\)$ such that:
\begin{equation}\label{eq:gbdd}
    |g|\leq 1,
\end{equation}
\begin{equation}\label{eq:g'bdd}
    |g'|\lesssim h^{-1/2},
\end{equation}
and
\begin{equation}\label{eq:qwt}
    \rr\(X\)+ hH_{\im q}g\(X\) + C_{2}h\gtrsim h^{2/3}\lambda_{q}\(X\)^{1/3},\quad X\in\R^{2n}
\end{equation}
for some $C_{2}>0$ and all $h>0$ sufficiently small.
\end{lemma}
\begin{proof}
Define $G\(X\)$ in the region where $\lambda_{p}\(X\)\geq h$ by
\[G=\epsilon h^{-1/3}\frac{H_{\im p}\rr p}{\lambda_{p}^{2/3}}\psi\(\frac{\rr p}{h^{2/3}\lambda_{p}^{1/3}}\),\]
with $\epsilon>0$, to be chosen later, independent of $h$.
Then we define $g$ by
\[g=\(1-\psi\(\frac{\lambda_{p}}{2h}\)\)G.\]
We will first show that the following hold where $G$ is defined:
\[|G|\lesssim \epsilon,\]
\[|G'|\lesssim \epsilon h^{-1/2},\]
and
\begin{equation}\label{eq:pineq}
    \rr p+hH_{\im p}G\gtrsim h^{2/3}\lambda_{p}^{1/3}.
\end{equation}
Then we will use these to show the desired properties for $g$.
Proving this works almost identically to the $A=0$ case from \cite{mine}.
The support of $G$ is contained in the region where $\rr p\leq h^{2/3} \lambda_{p}^{1/3}$, so we see that since $\psi \leq 1$ 
\[|G\(X\)| \leq \epsilon h^{-1/3}\frac{2|V_{2}'\(x\)||\xi-A|}{\lambda_{p}\(X\)^{2/3}}\psi\(\frac{\rr p}{h^{2/3}\lambda_{p}^{1/3}}\)\]
\[
\lesssim \epsilon h^{-1/3}\frac{\lambda_{p}^{1/2}\(h^{1/3} \lambda_{p}^{1/6}\)}{\lambda_{p}^{2/3}}
\lesssim \epsilon.
\]
Using \eqref{eq:v''s1}, \eqref{eq:A'}, and \eqref{eq:v1'v1} we get that
\begin{equation}\label{eq:rep'rep}
    |\rr p'|\lesssim \(\rr p\)^{1/2},
\end{equation}
and
\begin{equation}\label{eq:lp'lp}
    |\lambda_{p}'|\lesssim |\rr p'|+ |V_{2}'|\lesssim \lambda_{p}^{1/2}.
\end{equation}
Then to estimate $|G'|$ we have the following estimates on the support of $G$, using the above and that $|\xi-A|\leq h^{1/3}\lambda_{p}^{1/6}$ in this region:
\begin{equation}\label{eq:est1} \left|\frac{H_{\im p}\rr p}{\lambda_{p}^{2/3}}\right| = O\(h^{1/3}\), \end{equation}
\begin{equation}\label{eq:est2}\left|\partial^{\alpha} \frac{H_{\im p}\rr p}{\lambda_{p}^{2/3}}\right| = 
O\(\lambda_{p}^{-1/6}\) = O\(h^{-1/6}\) ,\quad  |\alpha|=1,\end{equation}
\begin{equation}\label{eq:est3}\left|\partial^{\alpha}\(\psi\(\frac{\rr p}{h^{2/3} \lambda_{p}^{1/3}}\)\)\right| = \end{equation}
\[= O\(h^{-1/3}\lambda_{p}^{-1/6}+\lambda_{p}^{-1/2}\) = O\(h^{-1/2}\),\quad  |\alpha|=1.\]
Thus by \eqref{eq:est1}, \eqref{eq:est2}, and \eqref{eq:est3},
$$|G'| = \epsilon h^{-1/3}\( O\(h^{-1/6}\) + O\(h^{1/3} h^{-1/2}\) \) = O\(\epsilon h^{-1/2}\),$$
which verifies that $|G'|\lesssim \epsilon h^{-1/2}$.

Now we shall attain \eqref{eq:pineq} in the region where 
\[\rr p \leq\frac{1}{4}h^{2/3} \lambda_{p}^{1/3}\leq \frac{1}{4}\lambda_{p},\] 
and so $2|V_{2}'\(x\)|^{2}\geq \frac{3}{4}\lambda_{p}\(X\)$.
In this region $\psi\(\frac{\rr p}{h^{2/3}\lambda_{p}^{1/3}}\)\equiv 1$, and so $G =\epsilon h^{-1/3}\frac{H_{\im p}\rr p}{\lambda_{p}^{2/3}}$.
Now we get
\begin{equation}\label{eq:hgv}
H_{\im p}G=\epsilon h^{-1/3}\(\frac{2|V_{2}'\(x\)|^{2}}{\lambda_{p}^{2/3}}- \frac{8 \(V_{2}'\(x\)\cdot \(\xi-A\)\)^{2}}{3 \lambda_{p}^{5/3}}\).
\end{equation}
Thus 
$$
{\rm Re}\,  p\(X\) + h H_{\im p} G\(X\) = \rr p\(X\) +\epsilon h^{2/3}\(\frac{2|V_{2}'\(x\)|^{2}}{\lambda_{p}^{2/3}} 
-\frac{8\(V_{2}'\(x\)\cdot\(\xi-A\)\)^{2}}{3 \lambda_{p}^{5/3}}\)$$
$$\geq \rr p\(X\) + \epsilon h^{2/3}\(\frac{2|V_{2}'\(x\)|^2}{\lambda_{p}^{2/3}} - \frac{2|V_{2}'\(x\)|^{2}}{3 \lambda_{p}^{2/3}}\)
\geq \epsilon h^{2/3}\frac{4|V_{2}'\(x\)|^2}{3\lambda_{p}^{2/3}}$$
$$\geq \frac{1}{2}\epsilon h^{2/3} \lambda_{p}^{1/3}.$$
It remains to show the bound in the region where $\rr p\geq \frac{1}{4}h^{2/3} \lambda_{p}^{1/3}$.
Using \eqref{eq:est1}, \eqref{eq:est2}, and \eqref{eq:est3} we get that 
\begin{multline*}
|hH_{\im p}G|\leq\epsilon h^{2/3}\lambda_{p}^{1/2} O\(\lambda_{p}^{-1/6} \) \\
+ \epsilon h^{2/3}\lambda_{p}^{1/2}O\(h^{1/3}\(h^{-1/3} \lambda_{p}^{-1/6} + \lambda_{p}^{-1/2}\)\)\\
=O\(\epsilon h^{2/3}\lambda_{p}^{1/3}\).
\end{multline*}
Now for $\epsilon$ sufficiently small we get
$$\rr p + hH_{\im p}G \gtrsim h^{2/3}\lambda_{p}^{1/3} - O\(\epsilon  h^{2/3}\lambda_{p}^{1/3}\) \gtrsim  h^{2/3}\lambda_{p}^{1/3}.$$
Thus $G$ has all of the claimed properties, and we will now show the corresponding properties for $g$.

As $|G|\lesssim \epsilon$, the same holds for $g$, and we now fix the value of $\epsilon$ by choosing it small enough such that $|g|\leq 1$.
To check that $|g'|\lesssim h^{-1/2}$ we 
note that in the region where $\lambda_{p} \geq 2h$, we have already shown that it holds as $g=G$ there. When $\lambda_{p}<h$
it holds trivially as $g=0$ there.
In the intermediate region where $h\leq \lambda_{p} \leq 2h$, we see, by using \eqref{eq:lp'lp}, 
\[|g'|\lesssim |G'|+\left|\frac{\lambda_{p}'}{2h}G\right| \lesssim h^{-1/2}+ \frac{\lambda_{p}^1/2}{h}|G|\lesssim h^{-1/2},\]
thus verifying \eqref{eq:g'bdd}.
To attain \eqref{eq:qwt} we will show the corresponding bound with $q$ replaced by $p$,
\begin{equation}\label{eq:pwt}
    \rr p\(X\) +hH_{\im p}g\(X\) + C_{2}h\gtrsim h^{2/3}\lambda_{p}\(X\)^{1/3},\quad \forall X\in\R^{2n},
\end{equation}
and then show that this implies desired inequality with $q$.
To check \eqref{eq:pwt} we first note that we have already shown that it holds where $\lambda_{p}\geq 2h$ in \eqref{eq:pineq}. 
When $\lambda_{p}<2h$ we use \eqref{eq:g'bdd} to see that
\begin{equation}\label{eq:himpg}
    |hH_{\im p}g|\lesssim h^{1/2}|V_{2}'|\lesssim h.
\end{equation}
Thus, choosing $C_{2}$ sufficiently large,
\[\rr p\(X\)+hH_{\im p}g\(x\)+C_{2}h\gtrsim h\gtrsim h^{2/3}\lambda_{p}^{1/3}.\]
Now it remains to show that this implies the related fact for $q$, \eqref{eq:qwt}.
In the region where $p=q$ this implication is trivial. In the region where $p$ and $q$ can differ, 
i.e. where $|\xi|\geq 2R\langle x\rangle$, we recall from \eqref{eq:pqell} that
\[\rr q\sim \lambda_{q}\sim |\xi|^{2}\gtrsim 1.\]
Then, as
\[|hH_{\im q}g| \lesssim h^{1/2}|V_{2}'|\lesssim h^{1/2}\lambda_{q}^{1/2},\]
we can see that $\rr q$ is the dominant term on the left-hand side of $\eqref{eq:qwt}$ in this region, 
 and so 
 \[\rr q-O\(h^{1/2}\(\rr q\)^{1/2}\)+ C_{2}h\gtrsim h^{2/3}\(\rr q\)^{1/3}\gtrsim h^{2/3}\lambda_{q}^{1/3}\]
 for all $h$ sufficiently small.
\end{proof}

Lemmas \ref{q''lemma} and \ref{gexists}, and the fact that $\rr q\geq 0$ are the main ingredients needed to adapt the proof from Section 4 of \cite{mine} to prove
\eqref{eq:qest}.

\section{Wick Quantization}
In addition to the Weyl quantization, we will also work with pseudodifferential operators in the Wick quantization. 
Here we provide a brief summary of the relevant properties. More detail can be found in \cite{bc} and \cite{wicketc}.
Let $Y=\(y,\eta\)\in \R^{2n}$. Define $\phi_{Y}$ by
\[\phi_{Y}\(x\)= \pi^{-n/4} e^{-\frac{1}{2}|x-y|^{2}}e^{i\(x-y\)\cdot \eta},\quad \|\phi_{Y}\|_{L^{2}_{x}}=1.\]
Then define
\[\Pi_{Y}u\(x\)=\(u,\phi_{Y}\)\phi_{Y}\(x\),\]
where $\(\cdot ,\cdot\)$ denotes the $L^{2}$ scalar product.
Then for $a\in \Sch'\(\R^{2n}\)$ and $u\in \Sch\(\R^{n}\)$ we can define $a^{Wick}:\Sch\(\R^{n}\)\rightarrow\Sch\(\R^{n}\)$ by
\[a^{Wick}u:=a_{Y}\(\Pi_{Y}u\).\]
To see that this indeed maps to $\Sch\(\R^{n}\)$ we observe that for $u\in\Sch\(\R^{n}\)$
\[\(u,\phi_{Y}\)\in\Sch\(\R^{2n}_{\(y,\eta\)}\),\]
which is verified in Proposition 3.1.6 of \cite{mart}. 
So $\Pi_{Y}u\in\Sch\(\R^{3n}_{\(y,\eta,x\)}\)$, and then applying $a$ in the first two variables leaves $a_{Y}\(\Pi_{Y}u\)\in\Sch\(\R^{n}_{x}\)$.
It follows shortly from the definition that for a symbol $a\in\Sch'\(\R^{2n}\)$ and all $u\in \mathcal{S}\(\mathbb{R}^{n}\)$
\begin{equation}\label{eq:pos}a\geq 0 \Rightarrow \( a^{Wick}u,u\) \geq 0,\end{equation}
and, just as for the Weyl quantization, formal $L^{2}$ adjoints are attained by taking the complex conjugate of the symbol.
\begin{equation}\label{eq:wadj}
\(a^{Wick}\)^{*}=\(\overline{a}\)^{Wick}.
\end{equation}
Furthermore, if $a\in L^{\infty}$, then $a^{Wick}: L^{2}\rightarrow L^{2}$ and
\begin{equation}\label{eq:l2bd}
\|a^{Wick}\|_{L^{2}\rightarrow L^{2}} \leq \|a\|_{L^{\infty}}.    
\end{equation}
We also can relate the Wick quantization of a symbol $a\in S\(m\)$, for some order function $m$, to its Weyl quantization by
\begin{equation}\label{eq:ww}
a^{Wick}= a^{w} + r\(a\)^{w},
\end{equation}
where
\begin{equation}\label{eq:remainder}
r\(a\)\(X\)=\(\pi\)^{-n/2}\int_{0}^{1}\int_{\mathbb{R}^{2n}}\(1-t\)a''\(X+tY\)Y\cdot Y e^{-|Y|^2}dY dt.
\end{equation}
For smooth symbols $a$ and $b$ with
$a\in L^{\infty}\(\R^{2n}\)$ and $\partial^{\alpha}b \in L^{\infty}\(\R^{2n}\)$ for $|\alpha| = 2$ we have the following composition formula
proven in \cite{absource},
\begin{equation}\label{eq:comp}
a^{Wick}b^{Wick} = \(ab - \frac{1}{2}a' \cdot b' + \frac{1}{2i}\left\{ a, b \right\}\)^{Wick} + R, 
\end{equation}
where $\|R\|_{L^2 \rightarrow L^2} \lesssim \|a\|_{L^{\infty}} \sup\limits_{|\alpha|=2}\|\partial^{\alpha}b\|_{L^{\infty}} $.
We will now use this to show that \eqref{eq:qest}, the desired $L^{2}$ lower bound for $q^{w}_{h}-z$ on $\Sch\(\R^{n}\)$, holds.

\section{Proving the Lower Bound for $q^{w}_{h}-z$}

Let $u\in \Sch\(\R^{n}\)$, and let $z\in\C$ satisfy the hypotheses of Theorem \ref{thm1}. 
We will start by using Wick symbol calculus to use $g^{Wick}$ as a bounded multiplier for $q^{Wick}$, which will be related back to $q^{w}$.
By \eqref{eq:wadj}, Wick operators with real symbols are formally self adjoint. Thus,
\begin{multline}\label{eq:c}
{\rm Re}\, \(\left[q\(\sqrt{h}X\)-z\right]^{Wick}u, \left[2-g\(\sqrt{h}X\)\right]^{Wick}u\) = \\
{\rm Re} \( \left[2-g\(\sqrt{h}X\)\right]^{Wick}\left[\, \(q\(\sqrt{h}X\)-z\)\right]^{Wick}u ,\, u\)= \\
 \({\rm Re}\( \left[2-g\(\sqrt{h}X\)\right]^{Wick}\left[\, \(q\(\sqrt{h}X\)-z\)\right]^{Wick}\)u ,\, u\).
\end{multline}
For any Wick symbol $a$ it is true that 
$$
{\rm Re}\, a^{Wick}=\frac{1}{2}\(a^{Wick}+\(a^{Wick}\)^{*}\)
=\frac{1}{2}\(a^{Wick}+\(\overline{a}\)^{Wick}\)=\({\rm Re}\, a\)^{Wick}.
$$
Using this fact, and the composition formula for the Wick quantization, \eqref{eq:wcomp},
\begin{equation}\label{eq:d}
{\rm Re}\, \(\left[2-g\(\sqrt{h}X\)\right]^{Wick}\left[q\(\sqrt{h}X\)-z\right]^{Wick} \)= 
\end{equation}
\begin{multline*}
{\rm Re}\, \bigg[\(2-g\(\sqrt{h}X\)\)\(q\(\sqrt{h}X\)-z\) +
\frac{1}{2}\nabla\(g\(\sqrt{h}X\)\)\cdot\nabla\(q\(\sqrt{h}X\)\)\\
- \frac{1}{2 i}\left\{g\(\sqrt{h}X\),q\(\sqrt{h}X\)\right\}\bigg]^{Wick} + S_{h}\\
=\bigg[\(2-g\(\sqrt{h}X\)\)\(\rr q\(\sqrt{h}X\) - {\rm Re}\, z\) \\
+ \frac{h}{2}g'\(\sqrt{h}X\)\cdot \rr q'\(\sqrt{h}X\) + \frac{h}{2}H_{V_{2}} g\(\sqrt{h}X\)\bigg]^{Wick} + S_{h},
\end{multline*}
where $\|S_{h}\|_{L^{2}\rightarrow L^{2}}=O\(h\)$, because $|g|\leq 1$ and $q''\in S\(1\)$ by Lemma \ref{q''lemma}.
Then, because $|\rr q'|\lesssim \(\rr q\)^{1/2}$ and, by \eqref{eq:g'bdd}, $|g'|\lesssim h^{-1/2}$ we have 
$$ \left|h g'\(\sqrt{h}X\)\cdot \rr q'\(\sqrt{h}X\) \right|\lesssim h^{1/2}\(\rr q\(\sqrt{h}X\)\)^{1/2}$$
$$
\lesssim rh + \frac{1}{r}\rr q\(\sqrt{h}X\),
$$
for arbitrary $r>0$. By taking $r$ large enough the $\frac{1}{r}\rr q\(\sqrt{h}X\)$ term can be absorbed by $\(2-g\(\sqrt{h}X\)\)\rr q\(\sqrt{h}X\)$.

By using \eqref{eq:qwt} we get that for some $C_{1}$, $C_{2} > 0$ and arbitrary $C_{0}>0$,
$$
\(2-g\(\sqrt{h}X\)\)\(\rr q\(\sqrt{h}X\) -{\rm Re}\, z\) 
$$
$$+ \frac{h}{2}g'\(\sqrt{h}X\)\cdot \rr q'\(\sqrt{h}X\) 
+ \frac{h}{2}H_{V_{2}}g\(\sqrt{h}X\)
$$
\begin{equation}\label{eq:a}
\gtrsim \rr q\(\sqrt{h}X\) - 3 max\(0, {\rm Re}\, z\) + \frac{h}{2}H_{V_{2}} g\(\sqrt{h}X\) + O\(h\)
\end{equation}
$$
\gtrsim h^{2/3}\lambda_{q}\(\sqrt{h}X\)^{1/3} - C_{1} max\(0, {\rm Re}\, z\) - C_{2}h
$$
$$
\gtrsim h^{2/3}\(\lambda_{q}\(\sqrt{h}X\)^{1/3}- 2C_{0}C_{1}y^{1/3}\) + C_{0}C_{1} h^{2/3}y^{1/3}
$$
$$
+\ C_{1}\(C_{0} h^{2/3}y^{1/3} - max\(0, {\rm Re}\, z\)\) - C_{2}h.
$$
As we required that ${\rm Re}\, z \leq C_{0} h^{2/3}y^{1/3}$ it follows that
\begin{equation}\label{eq:summary}
h^{2/3}\(\lambda_{q}\(\sqrt{h}X\)^{1/3} - 2C_{0}C_{1}y^{1/3}\) + C_{1}\(C_{0}h^{2/3}y^{1/3} - max\(0, {\rm Re}\, z\)\)
\end{equation}
$$
\geq -2C_{0}C_{1}h^{2/3}y^{1/3}\psi\(\frac{B \lambda_{q}\(\sqrt{h}X\)}{y}\),
$$
where
\begin{equation}\label{eq:bdef}
B=\frac{1}{\(4C_{0}C_{1}\)^3},
\end{equation}
and $\psi$ is the same cutoff as before.
Fix the value of $0<C_{0}\leq 1$ by choosing it small enough such that we can use that $|V_{2}\(x\)|-T\lesssim V_{1}\(x\)+|V_{2}'\(x\)|^{2}$ to get
\begin{equation}\label{eq:bval}
 |q\(X\)|-T \leq \frac{B \lambda_{q}\(X\)}{2},\quad X\in\mathbb{R}^{2n},
\end{equation}
which we will need later.
Substituting \eqref{eq:summary} into \eqref{eq:a} gives
\begin{equation}\label{eq:b}
\(2-g\(\sqrt{h}X\)\)\(\rr q\(\sqrt{h}X\)-{\rm Re}\, z\) + \frac{h}{2}g'\(\sqrt{h}X\)\cdot  \rr q'\(\sqrt{h}X\) 
\end{equation}
$$+ \frac{h}{2}H_{V_{2}} g\(\sqrt{h}X\)$$
$$\gtrsim -2C_{0}C_{1}h^{2/3}y^{1/3}\psi\(\frac{B \lambda_{q}\(\sqrt{h}X\)}{y}\) - C_{2}h + C_{0}C_{1}h^{2/3}y^{1/3}.
$$
Now  \eqref{eq:pos}, \eqref{eq:c}, \eqref{eq:d}, and \eqref{eq:b} imply that, for $h$ sufficiently small, $\rr z\leq C_{0}h^{2/3}y^{1/3}$, and some $C_{4}, C_{5}>0$
$$
{\rm Re}\, \([q\(\sqrt{h}X\)-z]^{Wick}u, [2-g\(\sqrt{h}X\)]^{Wick}u\)+ C_{4}h\|u\|^{2} +
$$
$$C_{5}h^{2/3}y^{1/3}\(\psi\(\frac{B\lambda_{q}\(\sqrt{h}X\)}{y}\)^{Wick}u, u\)
\gtrsim h^{2/3}y^{1/3}\|u\|^{2}_{L^2}.
$$
By the Cauchy-Schwarz inequality and \eqref{eq:l2bd} we get that
$$
\left\|[q\(\sqrt{h}X\)-z]^{Wick}u\right\| + h\|u\|  +
h^{2/3}y^{1/3}\left\|\psi\(\frac{B\lambda_{q}\(\sqrt{h}X\)}{y}\)^{Wick}u\right\|
$$
$$\gtrsim h^{2/3}y^{1/3}\|u\|
$$
Now, as $y\geq Mh$, we pick $M$ sufficiently large so that the $h\|u\|$ term can be absorbed by the right-hand side to get
\begin{equation}\label{eq:wickineq}
\left\|\left[q\(\sqrt{h}X\)-z\right]^{Wick}u\right\| + h^{2/3}y^{1/3}\left\|\psi\(\frac{B\lambda_{q}\(\sqrt{h}X\)}{y}\)^{Wick}u\right\|
\end{equation}
$$
\gtrsim h^{2/3}y^{1/3}\|u\|.
$$

We now want to get the same bound for the Weyl quantizations of these symbols.
Because of \eqref{eq:q''s1}, \eqref{eq:ww}, and the Calder\'{o}n-Vaillancourt Theorem \eqref{eq:cv},
\begin{equation}\label{eq:qw2w}
q\(\sqrt{h}X\)^{Wick}=q\(\sqrt{h}X\)^{w}+O_{L^{2}\rightarrow L^{2}}\(h\).
\end{equation}
In order to do the same for the $\psi$ term, we need bounds on its derivatives. 
To do this we will first note the following bounds on the derivatives of $\lambda_{q}$,
which follow from \eqref{eq:v''s1}, \eqref{eq:q''s1}, and that $\rr q\geq 0$.
\begin{equation}\label{eq:lambda'}
   |\partial^{\alpha}\lambda_{q}|\lesssim |\rr q'|+ |V_{2}'|\lesssim \lambda_{q}^{1/2},\quad |\alpha|=1.
\end{equation}
\begin{equation}\label{eq:lambda''}
    |\partial^{\alpha}\lambda_{q}|\lesssim 1 + |V_{2}'|\lesssim 1 + \lambda_{q}^{1/2},\quad |\alpha|\geq 2.
\end{equation}    
We use these bounds to prove the following lemma.
\begin{lemma}\label{lemma2}
The derivatives of the $\psi$ term obey the following estimate.
\begin{equation}\label{eq:psisize}
\left|\partial^{\alpha}\(\psi\(\frac{B\lambda_{q}\(\sqrt{h}X\)}{y}\)\)\right|\lesssim
\frac{h^{1/2}}{y^{1/2}},\quad |\alpha|\geq 1.
\end{equation}
\end{lemma}
\begin{proof}
For $X\in\supp\(\psi\(\frac{B\lambda_{q}\(\sqrt{h}X\)}{y}\)\)$ we have
\[\lambda_{q}\(\sqrt{h}X\)\lesssim y,\]
and so, by \eqref{eq:lambda'}
\[\left|\partial^{\alpha}\(\frac{\lambda_{q}\(\sqrt{h}X\)}{y}\)\right|\lesssim \frac{h^{1/2}\lambda_{q}\(\sqrt{h}X\)^{1/2}}{y}
\lesssim \frac{h^{1/2}}{y^{1/2}},\quad |\alpha|=1,\]
and by \eqref{eq:lambda''}
\[\left|\partial^{\alpha}\(\frac{\lambda_{q}\(\sqrt{h}X\)}{y}\)\right|\lesssim h^{|\alpha|/2}\frac{1+\lambda_{q}\(\sqrt{h}X\)^{1/2}}{y}
\lesssim \frac{h}{y}+\frac{h}{y^{1/2}}\lesssim \frac{h^{1/2}}{y^{1/2}},\quad |\alpha|\geq 2.\]
We can express $\partial^{\alpha}\(\psi\(\frac{B\lambda_{q}\(\sqrt{h}X\)}{y}\)\)$ as a linear combination of terms of the form
\[\psi^{\(k\)}\(\frac{B\lambda_{q}\(\sqrt{h}X\)}{y}\)
\partial^{\gamma_{1}}\(\frac{\lambda_{q}\(\sqrt{h}X\)}{y}\) \ldots \partial^{\gamma_{k}}\(\frac{\lambda_{q}\(\sqrt{h}X\)}{y}\),\]
where $\alpha=\gamma_{1}+\ldots +\gamma_{k}$, $|\gamma_{i}|\geq 1$ for all $i$, $1\leq k\leq |\alpha|$.
Each such term is of size $O\(\(\frac{h}{y}\)^{k/2}\)$, proving the lemma.
\end{proof}
Thus
\begin{equation}\label{eq:psiw2w}
    \psi\(\frac{B\lambda_{q}\(\sqrt{h}X\)}{y}\)^{Wick}=\psi\(\frac{B\lambda_{q}\(\sqrt{h}X\)}{y}\)^{w}+O_{L^{2}\rightarrow L^{2}}\(\frac{h^{1/2}}{y^{1/2}}\).
\end{equation}
It then follows from \eqref{eq:wickineq}, \eqref{eq:qw2w}, and \eqref{eq:psiw2w}, taking $M$ sufficiently large, that
\begin{equation}\label{eq:weylineq}
\left\|\left[q\(\sqrt{h}X\)-z\right]^{w}u\right\| + h^{2/3}y^{1/3}\left\|\psi\(
\frac{B \lambda_{q}\(\sqrt{h}X\)}{y}\)^{w}u\right\|\gtrsim h^{2/3}y^{1/3}\|u\|.
\end{equation}
Now if we can show that the $\psi$ term can be absorbed into the other two we will get the desired inequality, \eqref{eq:qest}.
For the sake of brevity we will henceforth use the notation
\[\Psi\(X\):=\psi\(\frac{B\lambda_{q}\(\sqrt{h}X\)}{y}\).\]
Lemma \ref{lemma2} can then be rephrased as:
\[\Psi'\(X\)\in S\(\frac{h^{1/2}}{y^{1/2}}\).\]
Thus, by applying \eqref{eq:wcomp1} for the $h=1$ quantization we get that
$$
\(\Psi\(X\)^{w}\)^{2}=\Psi^{2}\(X\)^{w} + \frac{h}{y}R_{1}^{w},
$$
for some $R_{1}\in S\(1\)$.
Then by applying \eqref{eq:cv} and using that $\Psi^{w}$ is self adjoint we get
\begin{equation}\label{eq:psicomp}
\left\|\Psi\(X\)^{w}u\right\|^{2}_{L^2} = \(\Psi^2\(X\)^{w}u, u\) + O\(\frac{h}{y}\)\|u\|^{2}_{L^2}.
\end{equation}
To control the first term on the right-hand side we follow a method similar to Lemma 8.2 from \cite{hss} and Lemma 3 from \cite{mine}. 
\begin{lemma}\label{lemma3}
For all $u\in\Sch$, $h>0$ sufficiently small, and $z\in C$ with $|z|>KT+Mh$
$$\(\Psi^{2}\(X\)^{w}u, u\) \leq 
\(\(4\frac{\left|q\(\sqrt{h}X\)-z\right|^2}{y^2}\Psi^{2}\(X\)\)^{w}u, u\)
+O\(\frac{h^{1/2}}{y^{1/2}}\)\|u\|^{2}_{L^2}.$$
\end{lemma}
\begin{proof}
Recalling \eqref{eq:bval}, we see that for $X\in\supp\(\Psi\)$ 
\begin{equation}\label{eq:qonpsi}
\left|q\(\sqrt{h}X\)\right|-T\leq \frac{B\lambda_{q}\(\sqrt{h}X\)}{2} \leq \frac{y}{2}.
\end{equation}
This property is precisely what condition \eqref{eq:vv'} is needed for.
We then have
$$\frac{1}{y}\left|q\(\sqrt{h}X\)-z\right| \geq\frac{1}{y}\(|z|-\left|q\(\sqrt{h}X\)\right|\)$$

$$=\frac{1}{y}\(y+T-\left|q\(\sqrt{h}X\)\right|\)\geq \frac{1}{2},$$
and so
\begin{equation}\label{eq:q}
\Psi^{2}\(X\) \leq
4\frac{\left|q\(\sqrt{h}X\)-z\right|^2}{y^2}\Psi^{2}\(X\),\quad X\in\R^{2n}.
\end{equation}
Let 
\begin{equation}\label{eq:qdef}
Q\(X\)= 4\frac{\left|q\(\sqrt{h}X\)-z\right|^2}{y^2}\Psi^{2}\(X\)
-\Psi^{2}\(X\)\geq 0.
\end{equation}
By \eqref{eq:pos}, \eqref{eq:ww}, and \eqref{eq:remainder} we get that
\begin{equation}\label{eq:qgard}
\(Q^{w}\(x,D_{x}\)u,u\)_{L^2}\ +\quad 
\end{equation}
\[\left\|\pi^{-n/2}\(\int_{0}^{1}\int_{\mathbb{R}^{2n}}\(1-t\)Q''\(X+tY\)Y\cdot Y e^{-|Y|^2}dY dt\)^{w}u\right\|
\left\|u\right\|\geq 0.\]
To estimate the second term, \eqref{eq:cv} implies that we need to estimate the derivatives of order two and higher of $Q$.

As $|z|>KT+Mh$ and $K>1$,
\[y=|z|-T>\(K-1\)T\gtrsim T.\]
So, for $X\in\supp\(\Psi\)$, using \eqref{eq:qonpsi}, $y\gtrsim T$, and $y\gtrsim |z|$, we get the following
\begin{equation}\label{eq:qovery}
    \left| \frac{q\(\sqrt{h}X\)-z}{y}\right|\lesssim \frac{1}{y}\(y+T+|z|\)\lesssim 1.
\end{equation}
For such $X$, using that $|\rr q'|\lesssim \(\rr q\)^{1/2}$, we also have
\begin{equation}\label{eq:qy'}
\left|\partial^{\alpha}\frac{q\(\sqrt{h}X\)-z}{y} \right| \lesssim \frac{h^{1/2}}{y}\lambda_{q}\(\sqrt{h}X\)^{1/2}
\lesssim \frac{h^{1/2}}{y^{1/2}},\, |\alpha|=1
\end{equation}
and
\begin{equation}\label{eq:qy''}
\left| \partial ^{\alpha}\frac{q\(\sqrt{h}X\)-z}{y} \right| \lesssim \frac{h^{|\alpha|/2}}{y},\, |\alpha|\geq 2.
\end{equation}
By the above and \eqref{eq:psisize}, for $|\alpha|\geq 1$,
\[\left| \partial^{\alpha}Q\(X\) \right| \lesssim \frac{h^{1/2}}{y^{1/2}},\quad X\in\R^{2n}.\]
Thus by applying the Calder\'{o}n-Vaillancourt theorem \eqref{eq:cv} we can bound the latter term of \eqref{eq:qgard} as follows, 
$$
\left\|\(\int_{0}^{1}\int_{\mathbb{R}^{2n}}\(1-t\)Q''\(X+tY\)Y\cdot Y e^{-|Y|^2}2^{n}dY dt\)^{w}u\right\|
\lesssim \frac{h^{1/2}}{y^{1/2}}\|u\|.
$$
Therefore \eqref{eq:qgard} implies a variant of the sharp G\r{a}rding inequality (cf. Theorem 4.32 of \cite{zw}) for $Q$,
$$
\(Q^{w}\(x,D_{x}\)u,u\)_{L^2} + O\(\frac{h^{1/2}}{y^{1/2}}\)\|u\|^{2}_{L^2} \geq 0.
$$
By \eqref{eq:qdef} we attain the statement in the lemma.
\end{proof}
Combining \eqref{eq:psicomp} and Lemma \ref{lemma3} we get that
\begin{equation}\label{eq:psitoqpsi}
    \|\Psi^{w}u\|^{2}\leq \(\(4\frac{\left|q\(\sqrt{h}X\)-z\right|^2}{y^2}\Psi^{2}\(X\)\)^{w}u, u\)
+O\(\frac{h^{1/2}}{y^{1/2}}\)\|u\|^{2}_{L^2}.
\end{equation}
Finally, we have to understand the first term on the right side of \eqref{eq:psitoqpsi}. 
The estimates \eqref{eq:psisize}, \eqref{eq:qovery}, \eqref{eq:qy'}, and \eqref{eq:qy''} imply that
$$
\partial^{\alpha}\(\frac{\(q\(\sqrt{h}X\)-z\)}{y}\Psi\(X\)\)
=O\(\(\frac{h}{y}\)^{1/2}\),\, |\alpha|\geq 1.
$$
Then by applying \eqref{eq:wcomp1} we get
\begin{multline*}
4\frac{\left|q\(\sqrt{h}X\)-z\right|^2}{y^2}\Psi^{2}\(X\)\\
=4 \(\frac{\overline{\(q\(\sqrt{h}X\)-z\)}}{y}\Psi\(X\)\#
\frac{\(q\(\sqrt{h}X\)-z\)}{y}\Psi\(X\)\)
+\frac{h}{y}R_{2},
\end{multline*}
where $R_{2}\in S\(1\)$.
We also similarly get from \eqref{eq:psisize}, \eqref{eq:qy'}, \eqref{eq:qy''} and \eqref{eq:wcomp1} that
$$
\Psi\(X\)\# \frac{\(q\(\sqrt{h}X\)-z\)}{y} 
= \frac{\(q\(\sqrt{h}X\)-z\)}{y}\Psi\(X\) + \frac{h}{y}R_{3},
$$
for $R_{3}\in S\(1\)$. 

Now, using this, \eqref{eq:psitoqpsi}, the fact that $\frac{h}{y}\leq \frac{1}{M}$, and \eqref{eq:cv},
we can conclude that
$$
\left\Vert\Psi\(X\)^{w}u\right\Vert^{2}_{L^2}\lesssim \left\Vert \(\Psi\(X\)\)^{w} 
\(\frac{\(q\(\sqrt{h}X\)-z\)}{y}\)^{w}u\right\Vert^{2}_{L^2} + O\(\frac{h^{1/2}}{y^{1/2}}\)\|u\|^2_{L^2}
$$
$$
\lesssim \frac{1}{y^{2}}\left\|\left[q\(\sqrt{h}X\)-z\right]^{w}u\right\|^{2}_{L^2} + O\(\frac{1}{M^{1/2}}\)\|u\|^2_{L^2}.
$$

Plugging this into \eqref{eq:weylineq} we get
\begin{multline*}
\left\|\left[q\(\sqrt{h}X\)-z\right]^{w}u\right\| + \frac{h^{2/3}}{y^{2/3}}\left\|\left[q\(\sqrt{h}X\)-z\right]^{w}u\right\|\\
+ O\(\frac{1}{M^{1/4}}\)h^{2/3}y^{1/3}\|u\|
\gtrsim h^{2/3}y^{1/3}\|u\|.
\end{multline*}
Then taking $M$ sufficiently large yields
$$
\left\|\left[q\(\sqrt{h}X\)-z\right]^{w}u\right\|\gtrsim h^{2/3}y^{1/3}\|u\|.
$$
Finally, by making the symplectic change of coordinates $x\rightarrow \frac{x}{\sqrt{h}}$, $\xi\rightarrow \sqrt{h}\xi$
we obtain the desired estimate,
$$
\left\|\(q^{w}\(x, hD_{x}\) - z\)u\right\|\gtrsim h^{2/3}y^{1/3}\|u\|.
$$
The results of this section can summarized by the following.
\begin{proposition}
For any $K>1$ there exists constants $0<C_{0}\leq 1$, $M\geq 2$, and $h_{0}>0$ such that for all $z\in\C$
with $|z|\geq KT+Mh$ and $\rr z\leq C_{0}h^{2/3}y^{1/3}$ and all $0<h\leq h_{0}$,
\[\|\(q_{h}^{w}-z\)u\|\gtrsim h^{2/3}y^{1/3}\|u\|,\quad u\in\Sch.\]
\end{proposition}
Thus, by Corollary \ref{qthenp} the same holds for $p_{h}^{w}$.
All that remains to complete the proof of Theorem \ref{thm1} is to extend this lower bound to the maximal domain of $p_{h}^{w}$ so that we may conclude the corresponding upper bound for its resolvent.
\section{Attaining the Resolvent Estimate}
We will use the following to finish the proof of Theorem \ref{thm1}.
\begin{proposition}
Let $a\in C^{\infty}\(\R^{2n}\)$ with 
\begin{equation}\label{eq:a'sx}
    a'\in S\(\langle X\rangle\).
\end{equation}
Then the $L^{2}$-graph closure of $a^{w}$ on $\Sch$ has the maximal domain $D_{max}:= \left\{u\in L^{2} : a^{w}u\in L^{2}\right\}$.
\end{proposition}
\begin{proof}
To show that the graph closure of $a^{w}\(x,D_{x}\)$ on $\Sch\(\R^{n}\)$ has domain 
$D_{max}$ we follow a method
from H\"{o}rmander found in \cite{horm}. Let $\chi_{\delta}\in\Sch\(\R^{2n}\)$ be a family of symbols parametrized by $\delta>0$ such that
$\chi_{\delta}^{w}: L^{2}\rightarrow \Sch$ is a bounded family of operators
with
$\chi_{\delta}^{w}u\rightarrow u$ in $L^2$ as $\delta\rightarrow 0$ for all $u\in L^{2}$.
If 
\begin{equation}\label{eq:gc}
\(a^{w}\chi_{\delta}^{w}-\chi_{\delta}^{w}a^{w}\)u \rightarrow 0
\end{equation}
in $L^2$ as $\delta\rightarrow 0$ for all $u\in D_{max}$
then $u_{\delta}:=\chi_{\delta}^{w}u$ is a sequence of functions in $\Sch$ converging to $u$ and
$a^{w}u_{\delta}\rightarrow a^{w}u$ in $L^{2}$, thus the domain of the graph closure of $a^{w}$ is $D_{max}$.

To accomplish this, let $\phi\in C_{c}^{\infty}\(\R^{n},[0,1]\)$ be a cutoff function with $\phi\(x\)=1$ for $x$ in a neighborhood of $0$.
Then define
\begin{equation*}
    \chi_{\delta}= \(\phi\(\delta x\)\phi\(\delta \xi\)\).
\end{equation*}
We then have that $\chi_{\delta}^{w}:L^{2}\rightarrow \Sch$ and $\chi_{\delta}^{w}u\rightarrow u$ in $L^2$ as $\delta\rightarrow 0$ for all $u\in L^2$
as desired, which is quick to verify using Weyl calculus and Parseval's theorem.
We then need to check \eqref{eq:gc}. This can be accomplished using some Weyl symbol calculus for the commutator 
$[a^{w},\chi_{\delta}^{w}]$.
For $X\in\supp\(\chi_{\delta}\)$ it holds that $|X|\lesssim \delta^{-1}$. We also have that $|\partial^{\alpha}\chi_{\delta}|\lesssim \delta^{2}$ 
for $|\alpha|\geq 2$, 
and so $\chi_{\delta}''\in S\(\delta\langle X\rangle^{-1}\)$ uniformly in $\delta$.
    Using this and \eqref{eq:a'sx} it follows that for some $R_{4}\in S\(1\)$ we have
\begin{equation*}
[a^{w},\chi_{\delta}^{w}]=\(-i\left\{a\(x,\xi\),\phi\(\delta x\)\phi\(\delta \xi\)\right\}+ \delta R_{4}\)^{w}
\end{equation*}
\begin{equation}\label{eq:123}
=-i\delta\(\partial_{\xi}a \cdot \phi'\(\delta x\)\phi\(\delta \xi\)\)^{w}
+i\delta\(\partial_{x}a \cdot \phi'\(\delta \xi\)\phi\(\delta x\)\)^{w} u +\delta R_{4}^{w}
\end{equation}
$$
=I + II + III.
$$
Using \eqref{eq:a'sx} and the fact that $|X|\lesssim \delta^{-1}$ on $\supp\(\phi\(\delta x\)\phi\(\delta \xi\)\)$ we get
\begin{equation}\label{eq:achix}
    \left|\delta \partial^{\alpha} \(\partial_{\xi}a\cdot \phi'\(\delta x\)\phi\(\delta \xi\)\)\right|=O\(1\),\quad \forall \alpha
\end{equation}
and
\begin{equation}\label{eq:achixi}
\left|\delta \partial^{\alpha} \(\partial_{x}a\cdot \phi'\(\delta \xi\)\phi\(\delta x\)\)\right|=O\(1\),\quad \forall \alpha.
\end{equation}
Thus by \eqref{eq:cv}
\[\left\|[a^{w},\chi_{\delta}^{w}]\right\|_{L^{2}\rightarrow L^{2}}=O\(1\). \]
It thus suffices to show that $[a^{w},\chi_{\delta}^{w}]u\rightarrow 0$ for all $u$ in a dense subset of $L^{2}$.
Term $III$ is easily dealt with because as $\delta\rightarrow 0$, 
\[\|III u\|=O\(\delta\)\|u\|\rightarrow 0,\quad u\in L^{2}.\]
To deal with terms $I$ and $II$, let $u\in C_{c}^{\infty}\(\R^{n}\)$.
Note that the Weyl symbol from $II$ is supported where $|\xi|\sim \delta^{-1}$ and is in $S\(1\)$ by \eqref{eq:achixi}.
Then, using the definition of the Weyl quantization and integration by parts, 
\[II u= \frac{i\delta}{\(2\pi\)^{n}}\int_{\mathbb{R}^{2n}} e^{i\(x-y\)\cdot\xi}
\(\partial_{x}a\)\(\frac{x+y}{2}, \xi\)\cdot \phi'\(\delta \xi\)\phi\(\frac{\delta\(x+y\)}{2}\) u\(y\) dy d\xi \]
\[=\frac{-i\delta}{\(2\pi\)^{n}}\int_{\mathbb{R}^{2n}}\frac{e^{i\(x-y\)\cdot\xi}}{|\xi|^{2}} \Delta_{y}\(\(\partial_{x}a\)\(\frac{x+y}{2}, \xi\)\cdot \phi'\(\delta \xi\)\phi\(\frac{\delta\(x+y\)}{2}\) u\(y\)\) dy d\xi.\]
\[=\sum\limits_{|\alpha|\leq 2}b_{\alpha}^{w}\partial^{\alpha}u,\]
where each $b_{\alpha}\in S\(\delta^{2}\)$. Thus
\[\|IIu\|\lesssim \delta^{2}\rightarrow 0.\]

Similarly, the Weyl symbol in $I$ is in $S\(1\)$ by \eqref{eq:achix}, and
\[I u=\frac{-i\delta}{\(2\pi\)^{n}}\int_{\mathbb{R}^{2n}} e^{i\(x-y\)\cdot\xi}
\(\partial_{\xi}a\)\(\frac{x+y}{2}, \xi\)\cdot \phi'\(\frac{\delta\(x+y\)}{2}\)\phi\(\delta \xi\) u\(y\) dy d\xi \]
\[=\frac{i\delta}{\(2\pi\)^{n}}\int_{\mathbb{R}^{2n}}\frac{e^{i\(x-y\)\cdot\xi}}{|x-y|^{2}} \Delta_{\xi}\(\(\partial_{\xi}a\)\(\frac{x+y}{2}, \xi\)\cdot \phi'\(\frac{\delta\(x+y\)}{2}\)\phi\(\delta \xi\)\) u\(y\) dy d\xi.\]
The integrand is supported where $|x+y|\sim \delta^{-1}$ and $|y|\lesssim 1$. So for $\delta$ sufficiently small $|x-y|\sim \delta^{-1}$.
Thus
\[I u= \frac{1}{\(2\pi\)^{n}}\int_{\R^{2n}}e^{i\(x-y\)\cdot\xi}p\(x,y,\xi\)u\(y\)dy d\xi,\]
where $\partial^{\alpha}p=O\(\delta^{2}\)$ for all $\alpha$. Thus by Theorem 4.20 of \cite{zw}
\[\|Iu\|\lesssim \delta^{2}\rightarrow 0.\]

Therefore \eqref{eq:gc} holds, which tells us that the graph closure of $a^{w}$ on $\Sch$ 
has the domain $D_{max}$. 
\end{proof}
The above lemma applies to $p^{w}\(x,hD_{x}\)-z$ as we have $p'\in S\(\langle X\rangle\)$ from \eqref{eq:p'x}.
Let $P$ denote $p^{w}\(x,hD_{x}\)$ extended by graph closure to its maximal domain, $D\(P\):=\{u\in L^{2}:Pu\in L^{2}\}$. We then have that
\[\|\(P-z\)u\|\gtrsim h^{2/3}y^{1/3}\|u\|,\quad \forall u\in D\(P\),\]
and so $P-z$ is injective with closed range.
We can apply the same argument to the formal adjoint of $p^{w}-z$ on $\Sch$, $\overline{p}^{w}-\overline{z}=\(hD_{x}-A\)^{2}+\overline{V}\(x\)-\overline{z}$,
and we similarly get its graph closure, $\overline{P}-\overline{z}$, is also injective with maximal domain and closed range. 
Thus $\overline{P}=P^{*}$, and $P-z$ in invertible.
We then get the desired resolvent estimate,
\[\|\(P-z\)^{-1}u\|\lesssim h^{-2/3}\(|z|-T\)^{-1/3}\|u\|.\]

\textit{Acknowledgment.} I am very grateful to Michael Hitrik whose input was invaluable in the creation of this paper.

\end{document}